\documentclass{amsproc}

\usepackage{amsmath,amssymb,amsthm}
\usepackage[unicode,breaklinks=true,colorlinks=true]{hyperref}
\usepackage[dvipsnames]{xcolor}
\usepackage[normalem]{ulem} 
\usepackage{soul}

\numberwithin{equation}{section}

\def\indeq{\quad{}}           

\def\pv{\mathop{\mathrm{p.v.}}}

\definecolor{darkblue}{rgb}{0,0,0.7}

\newcommand{\norm}[1]{\| #1 \|}

\newcommand{\be}{\beta}
\newcommand{\de}{\delta}
\newcommand{\e}{\epsilon}
\newcommand{\ga}{{\gamma}}

\newcommand{\la}{\lambda}

\newcommand{\Om}{{\Omega}}

\newcommand{\si}{\sigma}
\newcommand{\td}{\tilde}

\renewcommand{\th}{\theta}

\newcommand{\R}{{\mathbb R }}\newcommand{\RR}{{\mathbb R }}
\newcommand{\N}{{\mathbb N}}

\newcommand{\pd}{{\partial}}
\newcommand{\nb}{{\nabla}}

\newcommand{\I}{\infty}
 
\renewcommand{\div}{\mathop{\mathrm{div}}}

\newcommand{\supp}{\mathop{\mathrm{supp}}}

\newcommand{\donothing}[1]{{}}

\newcommand{\EQ}[1]{\begin{equation}\begin{split} #1 \end{split}\end{equation}}
\newcommand{\EQN}[1]{\begin{equation*}\begin{split} #1 \end{split}\end{equation*}}

\DeclareMathOperator*{\esssup}{ess\,sup}

\makeatletter
\newcommand{\xRightarrow}[2][]{\ext@arrow 0359\Rightarrowfill@{#1}{#2}}
\makeatother

\newcommand{\loc}{\mathrm{loc}} 
\newcommand{\far}{\mathrm{far}} 
\newcommand{\near}{\mathrm{near}} 
\newcommand{\uloc}{\mathrm{uloc}}

\newtheorem{theorem}{Theorem}[section]
\newtheorem{lemma}[theorem]{Lemma}

\theoremstyle{definition}
\newtheorem{definition}[theorem]{Definition}

\theoremstyle{remark}

\numberwithin{equation}{section}

\newtheorem{corollary}[theorem]{Corollary}

\begin{document}
\title{The structure of weak  solutions to the Navier-Stokes equations} 

\author{Zachary Bradshaw}
\address{Department of Mathematical Sciences, University of Arkansas}
\email{zb002@uark.edu}
\thanks{ZB was supported in part by the  NSF grant DMS-2307097.}

\author{Igor Kukavica}
\address{Department of Mathematics,
	University of Southern California}
\email{kukavica@usc.edu}
\thanks{IK was supported in part by the NSF grant DMS-2205493.}

\date{\today}


\maketitle

\begin{abstract}  
The existence of  superfluous   solutions to the Navier-Stokes equations in the whole space implies that not all solutions with uniformly locally bounded energy satisfy a useful local pressure expansion. We prove that every weak solution in a parabolic uniformly local $L^2$ class can be obtained as a transgalilean transformation of a solution satisfying the local pressure expansion in a distributional sense. This gives a powerful representation theorem for a large class of solutions. We use this structure to obtain a  sufficient condition for the local pressure expansion.
\end{abstract}


\section{Introduction}\label{sec.intro}

The Navier-Stokes equations describe the evolution of a viscous incompressible fluid's velocity field $u$ and associated scalar pressure~$p$. In particular, $u$ and $p$ are required to satisfy
\EQ{\label{eq.NSE}
&\partial_tu-\Delta u +u\cdot\nabla u+\nabla p = 0,
\\& \nabla \cdot u=0,
}
in the sense of distributions.  For our purpose, \eqref{eq.NSE} is applied on $\R^3\times (0,T)$ where $0<T\leq \I$ and $u$ evolves from a prescribed, divergence-free initial data $u_0\colon \R^3\to \R^3$.

When working with weak solutions to the Navier-Stokes equations, it is often useful to have an explicit formula for the pressure in terms of the velocity field. If $u\in L^p$ for some $p<\I$, then the Riesz transforms are often used for this.  If $u$ is not decaying as $x\to \I$, e.g., if $u\in L^\I$, then this is less straightforward. One approach (e.g.~\cite{KiSe,JiaSverak-minimal,JiaSverak,KwTs,LR2,BT8}) is to seek a \emph{local pressure expansion} based on a formula of Fefferman~\cite{Feff}.\footnote{An alternative approach is to use Wolf's local pressure formula---see \cite{Wolf,ChEtAl,Kwon}.}
This  formula cannot hold for every non-decaying solution, as evidenced by the existence of  superfluous (also called parasitic)  solutions for which the formula is false.
In this paper, we show that non-decaying solutions in a general class including the superfluous solutions can be represented in terms of non-decaying solutions which have a local pressure expansion. We also identify a new sufficient condition for a non-decaying solution to satisfy the local pressure expansion.

Let  
\EQ{\label{Gij.def}
G_{ij}f=R_i R_j f= - \frac 13 \de_{ij} f(x) + \pv\int K_{ij}(x-y) f(y)dy,
}
with 
\[  K_{ij}(y) =
\pd_i \pd_j \frac1{4\pi |y|} = \frac {-\de_{ij}|y|^2 + 3 y_i y_j}{4\pi |y|^5}.
\]
For fixed
$x_0\in \R^3$ and $R>0$, denote $B=B_R(x_0)$,
and let $\th(x)$ be a cut-off function which equals $1$ on $B_{2R}(0)$ and vanishes off of~$B_{4R}(0)$.  Also, let 
\[
K_{ij}^{2R}(x) = K_{ij}(x)(1-\theta (x))
\]
and  
\EQ{\label{GijB.def}
G_{ij}^B f(x) =   - \frac 13 \de_{ij} f(x) +\lim_{\e\to 0} \int_{ {|x-y|>\e}} (K_{ij}(x-y)-K_{ij}^{2R} ( {x_0}-y)  )  f(y)\,dy. 
}
Note that, for $x\in B$, 
\begin{align} 
G_{ij}^B(f_{ij})(x)&= -\frac 1 3 \de_{ij}f_{ij} (x)+ \pv \int (K_{ij}(x-y)-K_{ij}^{2R}(x_0-y))f_{ij}(y)\,dy \nonumber
\\&
=  -\frac 1 3 \de_{ij}f_{ij}(x) +\pv \int K_{ij}(x-y) f_{ij}(y) \th(x_0-y)\,dy \nonumber
\\&\quad+ \int_{|x_0-y|\geq 2R} (K_{ij}(x-y)-K_{ij}^{2R}(x_0-y))f_{ij}(y)(1-\th(x_0-y)) \,dy \nonumber
\\&= (- \Delta^{-1}\div\div )_{ij}(f_{ij}   \th(x_0-\cdot) ) (x)  \nonumber
\\&\quad+\int  (K_{ij}(x-y)-K_{ij}(x_0-y))  (1- \th(x_0-y))f_{ij}(y) \,dy .\nonumber
\label{eq.pressureexpansion}
\end{align}
Unlike \eqref{Gij.def}, which requires some decay, the principle-value integral in \eqref{GijB.def} converges a.e.~if $f_{ij} \in L^\I(\R^n)$ due to the extra decay of the kernel in the far-field.  

Denote by $ L^2_\uloc(0,T)$ the class of vector fields $u\colon \R^3\times (0,T)\to \R^3$ for which 
\[
\esssup_{0<t<T} \|u(t)\|_{L^2_\uloc}<\I.
\]

The next definition describes a local pressure expansion. The assumptions included appear necessary to ensure $p\in L^1_\loc(\R^3\times (0,T))$.

\begin{definition}\label{def.local.pressure}
Assume that $u\in L^2_\uloc(0,T)\cap L^q_\loc(\R^3\times (0,T) )$, for some
$T>0$ and $q>2$, is a distributional solution to \eqref{eq.NSE}, where $p$ is the associated pressure.  We say that $p$ satisfies the local pressure expansion
if, for every $R>0$, $t\in (0,T)$, and $x_0\in \R^3$, there exists a constant $c_{x_0,R}(t)$ so that, for all $x\in B_R(x_0)$,
\[
p(x,t)= G_{ij}^{B_R(x_0)} (u_iu_j (t))(x) +c_{x_0,R}(t)
\]
in $L^{q/2}_\loc(  \R^3\times (0,T)  )$.
\end{definition}

If instead we worked with $u\in L^2_\uloc(0,T)$ without assuming higher integrability, then the local part of the local pressure expansion may not be defined in~$L^1_\loc$.  That is, if $u\in L^2$, then $(-\Delta^{-1}\div \div)_{ij}(u_iu_j \th)$ is not necessarily in $L^1_\loc$ because $u_iu_j\th\in L^1$ and the Riesz transforms are not bounded on~$L^1$.
This necessitates a weaker formulation of the local pressure expansion.
In what follows, $\th_R(x)=\Theta (x/R)$ where $\Theta\in C_c^\I$ is a fixed function  satisfying $\Theta = 1$ on $B_2(0)$ and $\supp\Theta \subset B_4(0)$. It follows that $\th_R =1$ on $B_{2R}(0)$ and $\supp \th_R \subset B_{4R}(0)$.

\begin{definition}\label{def.local.pressure2}
Assume that $u\in L^2_\uloc(0,T), $ for some $T>0$, is a distributional solution to \eqref{eq.NSE}, with $p\in \mathcal D'$ as the associated pressure.  We say that $  p$ satisfies the \textbf{distributional} local pressure expansion  
if, for every $R>0$ and $x_0\in \R^3$, there exists a spatially constant function of time $c_{x_0,R}(t)\in L^1(0,T)$ so that for every $\psi\in \mathcal D(B_R(x_0)\times (0,T))$, 
\EQ{\label{eq.press.distribution}
\int_0^T \langle p(x,t) - c_{x_0,R}(t),   \psi \rangle \,dt  &=\int_0^T \langle p_\near ,\psi\rangle \,dt +\int_0^T \langle p_\far ,\psi\rangle \,dt,
}
where 
\EQ{\notag
&\int_0^T \langle p_\near ,\psi\rangle \,dt
 :=\int_0^T\int u_i u_j \th_R(x_0-\cdot) (-\Delta^{-1} \div \div)_{ij} (\nb\cdot \psi)\,dx \,dt,
\\& \int_0^T \langle p_\far ,\psi\rangle \,dt:=  \int_0^T\int \int (K_{ij}(x-y) - K_{ij}(x_0-y))
\\&\indeq\indeq\indeq\indeq\indeq\indeq\indeq\indeq\indeq\indeq\indeq\indeq\indeq\indeq\indeq\times
(1-\th_R (x_0-y)) u_i u_j \,dy \nb \cdot \psi \,dx\,dt.
}
\end{definition}

The bracket notation $\langle T,f\rangle$ denotes the action of a distribution $T$ on a test function~$f$. We will use the same notation for distributions in $\mathcal D'(\R^3\times (0,T))$, $\mathcal D'(\R^3)$, and $\mathcal D'(0,T)$ and note that the meaning will be clear based on context.   We write $(-\Delta^{-1} \div \div)_{ij}=R_iR_j$, and we will use these notations interchangeably. 
In Section~\ref{sec.pressure}, we explicitly construct a distribution satisfying Definition~\ref{def.local.pressure2}.  The definitions for $p_\near$ and $p_\far$ make sense in $\mathcal D'(\R^3\times (0,T))$ whenever $u\in L^2_\uloc(0,T)$ (see \cite{BT7}). 

Definition~\ref{def.local.pressure2} was introduced in a paper by  Bradshaw and Tsai \cite{BT7}, where it is shown that certain weak solutions are mild if and only if their pressures satisfy the distributional local pressure expansion. This can be viewed as an alternative way to define the pressure when $u\in L^2_\uloc(0,T)$ compared to \cite[Ch. 11]{LR} where the Littlewood-Paley decomposition is used. The benefit is that it clearly extends the classical structure that the pressure has when $u$ is decaying, namely,
\EQ{\label{eq.pressure.classic}
p = R_i R_j(u_i u_j).
}
When $u$ is in $L^r(\R^3\times (0,T))_\loc$ for some $r>2$, the distributional local pressure expansion  agrees with the local pressure expansion in~$\mathcal D'$. When $u\in L^r(\R^3\times (0,T))$ for some $r>2$, then we furthermore have agreement with~\eqref{eq.pressure.classic}.

We will consider several classes of solutions, the most general of which is defined presently.

\begin{definition}[Weak solution]\label{def.weak.sol}  Assume $u_0\in L^1_\loc$ is divergence-free.
The pair $(u,p)$ is a weak solution to \eqref{eq.NSE} for $u_0$ if: 
\begin{enumerate}
\item $u,p\in \mathcal D'(\R^3\times (0,T))$ and, additionally, $u\in L^2_\uloc(0,T)$,
\item $(u,p)$ solves \eqref{eq.NSE} in $\mathcal D'(\R^3\times (0,T))$,
\item the limit
\[
\int_{0+}^T \langle p,\partial_k \psi \rangle  = \lim_{\e \to 0^+} \int_\e^T   \langle p,\partial_k  \psi \rangle \]         exists   for every $\psi\in C_0^\I( \R^3\times [0,T)  )$,
\item for each $k=1,2,3$, we have  
\EQ{\label{eq.weakForm}
\int_0^T\int  u_k(\partial_t \psi_k +\Delta \psi_k +u_j\partial_j \psi_k) \,dx\,dt+\int_{0+}^T \langle p,\partial_k \psi_k \rangle \,dt = - \int  u_{0k}\psi_k(\cdot,0) \,dx.
}
\end{enumerate}
\end{definition}

The last two items can be viewed a statement on the convergence of the solution to the initial data. Indeed, as will be shown in Section~3, they imply $u(t)\to u_0$ in $\mathcal D'(\R^3)$ as $t\to 0^+$. They are included because they are preserved under the transgalilean transformation that is at the heart of this paper.

Our first theorem states that any weak solution can be obtained by applying a transformation to a weak solution satisfying the distributional local pressure expansion.

\begin{theorem}[Structure of weak solutions]\label{thrm.structure}
Assume  $u_0\in L^1_\loc$ is divergence-free and $(u,p)$
is a weak  solution in the sense of Definition~\ref{def.weak.sol}.
Then, there exists a  weak  solution  $(\td u,\td p)$ so that $\td p $ satisfies the distributional local pressure expansion for  $\td u$ and
\[
\td u(x,t) = u(x-\Phi(t),t) +\phi(t),
\]
where 
\[
\phi(t)\in L^\I(0,T);\quad \lim_{t\to 0^+} \phi(t)=0,
\]
and 
\[
\Phi(t)=\int_0^t \phi(s)\,ds.
\]
\end{theorem}

This result builds upon an idea in~\cite{Kukavica}, where a similar representation is given for the large class of bounded solutions in terms of the smaller class of mild solutions. In \cite{KuVi},  this problem  is examined for a class of weak solutions inspired by the Koch-Tataru space.  All of these solutions are bounded at positive times. In contrast, Theorem~\ref{thrm.structure} allows for possibly singular solutions. In \cite{Kukavica} and \cite{KuVi}, the fact that the solutions are mild is emphasized and the pressure formula which comes from \cite{Feff} is dealt with implicitly. The results in \cite{BT7} show that these notions are essentially equivalent, even under the weakened hypotheses of Theorem~\ref{thrm.structure} (this was known in other contexts prior to \cite{BT7}, see e.g.~\cite{GIM} for bounded solutions and  \cite{LR} for a Littlewood-Paley based pressure formula).

If $(u,p)$ is a solution as in the statement,
then it is considered a genuine solution of the Navier–Stokes equations if 
$p$ satisfies the distributional local pressure expansion or, equivalently, is mild; otherwise, the solution is considered superfluous.

\bigskip 
Intuitively, if a solution exhibits any decay, then it cannot be a perturbation by a constant of another solution. Therefore, it should satisfy the distributional local pressure expansion. This can be made rigorous because $\phi(t)$ can be expressed in terms of~$u$. A consequence of this is the next theorem, which gives a sufficient condition for the local pressure expansion.

\begin{theorem}[Sufficient conditions for the local pressure expansion]\label{thrm.sufficient}Assume $u_0\in L^1_\loc$, which is divergence-free, and $u\in L^2_\loc(\R^3\times [0,T])$ satisfy
\EQ{
\lim_{R\to \I} \frac 1 {R^3} \int_{B_R(0)} |u_0|\,dx=0,
}
and 
\EQ{\label{cond.decay}
\lim_{R\to \I} \frac 1 {R^3} \int_0^T\int_{B_R(0)} |u|^2\,dx\,dt = 0.
}
If additionally the pair $(u,p)$ is a weak solution, then $p$ satisfies the distributional local pressure expansion.
\end{theorem}

There are several related sufficient conditions for the pressure in the literature. The most similar  is due to Lemari\'e-Rieusset.  In particular, \cite[Theorem~11.1.ii]{LR} states that if
\[
\lim_{R\to \I} \sup_{x_0\in \R^3} \frac 1 {R^3} \int_0^T\int_{B_R(x_0)} |u|^2\,dx\,dt = 0,
\]
and $u$ is a  ``uniform weak solution'' (as defined in \cite[Ch.~11]{LR}), then $u$ solves
\[
\partial_t u -\Delta u+\mathbb P\cdot\nb (u\otimes u)=0.
\] 
Then, \cite[Theorem~11.1.i]{LR} implies a pressure $P$ can be formulated using the Littlewood-Paley decomposition.  Our condition on the flow is weaker because it is  centered at~$0$.  However, the initial data is not mentioned in \cite[Theorem~11.1.i]{LR} whereas we require some decay.  The details of how to go from Lemari\'e-Rieusset's formula to the local pressure expansion under this generality has not been worked out explicitly. Theorem~\ref{thrm.sufficient} circumvents this issue without using the Littlewood-Paley theory and, together with \cite{BT7}, suggests these approaches are equivalent; see also~\cite{FL2}.

In conjunction with \cite[Theorem 1.5]{BT5}, this gives a sufficient condition for $u$ to be a mild solution provided the terms in the definition of mild solutions converge.

The local pressure expansion has been proven useful for studying a class of solutions introduced by Lemari\'e-Rieusset, and we take some care to apply our result to this context.
The following definition is motivated by those found in~\cite{LR,KiSe,JiaSverak-minimal,JiaSverak}.   Note that, at this point, we do not include an assumption on the structure of the pressure nor on the decay of the solution. 

\begin{definition}[Local Leray solution]\label{def:localLeray} A vector field $u\in L^2_{\loc}(\R^3\times [0,T))$ is a local Leray solution to \eqref{eq.NSE} with divergence-free initial data $u_0\in L^2_{\uloc}(\R^3)$ (denoted $u\in \mathcal N(u_0)$) if:
\begin{enumerate}
\item for some $p\in L^{3/2}_{\loc}(\R^3\times (0,T))$, the pair $(u,p)$ is a distributional solution to \eqref{eq.NSE},
\item for any $R>0$, the function $u$ satisfies
\begin{align}
  \begin{split}
    &\esssup_{0\leq t<R^2\wedge T}\,\sup_{x_0\in \R^3}\, \int_{B_R(x_0 )}\frac 1 2 |u(x,t)|^2\,dx
\\&\indeq
+ \sup_{x_0\in \R^3}\int_0^{R^2\wedge T}\int_{B_R(x_0)} |\nabla u(x,t)|^2\,dx \,dt<\infty,
  \end{split}
\notag
\end{align}
\item for all compact subsets $K$ of $\R^3$, we have $u(t)\to u_0$ in $L^2(K)$ as $t\to 0^+$,
\item $u$ is suitable in the sense of \cite{CKN}, i.e., for all cylinders $Q$ compactly supported in  $ \R^3\times(0,T )$ and all non-negative $\phi\in C_0^\infty (Q)$, we have  the \emph{local energy inequality}
\EQ{\label{CKN-LEI}
&
2\iint |\nabla u|^2\phi\,dx\,dt 
\\&\leq 
\iint |u|^2(\partial_t \phi + \Delta\phi )\,dx\,dt +\iint (|u|^2+2p)(u\cdot \nabla\phi)\,dx\,dt,
}
\item the function $t\mapsto \int u(x,t)\cdot w(x)\,dx$ is continuous on $[0,T)$ for any compactly supported $w\in L^2(\R^3)$.
\end{enumerate}
\end{definition} 

This class is not restrictive enough for some applications because it contains superfluous solutions.
A slightly stronger class of solutions is the following.

\begin{definition}[Local energy solution]
If $u\in \mathcal N(u_0)$ for some divergence-free $u_0\in L^2_\uloc$ and the associated pressure $p$ satisfies the local pressure expansion, 
then  $u$ is a  local energy solution. 
\end{definition}

Our naming convention is consistent with \cite{BT7}, which also includes a summary of the similarities and differences between the various definitions of ``local Leray solutions.''  Above, we assume that $p$ satisfies the local pressure expansion, which is stronger than the local pressure expansion as a distribution.

Given a local Leray solution, it is unclear how the evolution of the local energy relates to the initial data. This issue goes away if the solution is a local energy solution. In this case, Jia and \v Sver\'ak proved essentially the following estimate: ~If $u$ is a local energy solution with initial data $u_0\in L^2_\uloc$ and $r>0$, then
\begin{equation}\label{ineq.apriorilocal}
\esssup_{0\leq t \leq \sigma r^2}\sup_{x_0\in \RR^3} \int_{B_r(x_0)}\frac {|u|^2} 2 \,dx  + \sup_{x_0\in \RR^3}\int_0^{\sigma r^2}\int_{B_r(x_0)} |\nabla u|^2\,dx\,dt \leq
CA_0(r) ,
\end{equation}
where
\[
A_0(r)=rN^0_r= \sup_{x_0\in \R^3} \int_{B_r(x_0)} |u_0|^2 \,dx,
\] 
and
\begin{equation}\label{def.sigma}
\si=\sigma(r) =c_0\, \min\big\{(N^0_r)^{-2} , 1  \big\},
\end{equation}
for a small universal constant $c_0>0$. This estimate has been used critically in \cite{JiaSverak,KMT,BT5} and extended to a weighted setting in~\cite{BK1,BKO,BKT,BCT,FL1} which, interestingly, does not exactly overlap with $L^2_\uloc$.  Given the usefulness of this bound and its dependence on the local pressure expansion, it is important to have sufficient conditions under which  a local Leray solution is a local energy solution.  A sufficient condition for this is given in \cite{JiaSverak-minimal}, where it is stated that if $u$ is a local Leray solution and 
\[
\lim_{|x_0|\to \I} \int_0^{R^2}\int_{B_R(x_0)} |u|^2\,dx\,dt =0,
\]
then $u$ is a local energy solution. This has been proven explicitly in \cite{KMT} using ideas from~\cite{MaMiPr}. As a consequence of Theorem~\ref{thrm.sufficient}, we have an alternative  sufficient condition. 
\begin{corollary}\label{thrm.sufficient2} 
If $u_0\in L^2_\uloc$ is divergence-free, $u$ is a local Leray solution on $\R^3\times (0,T)$ with data $u_0$, and
\[
\lim_{R\to \I} \frac 1 {R^3} \bigg( \int_{B_R(0)}|u_0|\,dx + \int_0^T\int_{B_R(0)}|u(x,t)|^2\,dx \,dt     \bigg)=0,
\]
then $p$ satisfies the local pressure expansion. Moreover, $u$ is a local energy solution and \eqref{ineq.apriorilocal} holds. Additionally, from \cite[Theorem 1.5]{BT5}, it follows that $u$ is a mild solution.
\end{corollary}

This paper is organized as follows. In Section~\ref{sec.pressure}, we recall useful facts about the local pressure expansion from \cite{BT7} and prove several lemmas. Then, Sections \ref{sec.proof1} and \ref{sec.proof2} contain the proofs of Theorems \ref{thrm.structure} and \ref{thrm.sufficient}, respectively.

\section{The local pressure expansion} \label{sec.pressure}

To show that Definition~\ref{def.local.pressure2} is meaningful, we construct a distribution $p$ satisfying Definition~\ref{def.local.pressure2}. The details of this are contained in \cite{BT7}, and we only recall the main ideas. 

Let
$u\in L^2_\uloc(0,T) $, for some $T>0$, be given.
Fix $R>0$ and $x_0\in \R^3$, and let $B=B_R(x_0)$.
Consider the mapping
\EQN{ 
\psi \mapsto &\int_0^T \int u_i (x,t)u_j(x,t) \th_R(x_0-x) R_i R_j    \psi  (x,t)\,dx\,dt
\\&+ \int_0^T\int \int (K_{ij}(x-y) - K_{ij}(x_0-y))
\\&\indeq\indeq\indeq\indeq\indeq\indeq\indeq\indeq\indeq\indeq\indeq\times
(1-\th_R (x_0-y))( u_i u_j)(y,t) \,dy\,   \psi(x,t) \,dx\,dt
\\&=: \int_0^T \langle   \bar p_\near^B,\psi\rangle \,dt +\int_0^T \langle   \bar p_\far^B,\psi \rangle\,dt,
}
for $\psi\in {\mathcal D}(B\times (0,T))$,
where $\th_R$ is defined   in Section~1.
Then $ \bar p^B := \bar p_\near^B + \bar p_\far^B \in \mathcal D'(B\times (0,T))$ (the inclusion is proven in~\cite{BT7}).

We extend this to define a distribution in $\mathcal D'(\R^3\times (0,T))$ using the following recursive procedure:  
\EQ{\label{def.pressure.dist}
\begin{cases}\int_0^T \langle \bar p (x),\psi\rangle \,dt:=\int_0^T \langle \bar p^{B_1(0)} , \psi\rangle \,dt  & \text{ if }\psi \in \mathcal D( B_1(0)\times (0,T))
\\ \int_0^T\langle \bar p (x),\psi\rangle\,dt:=\int_0^T \langle \bar p^{B_n(0)} , \psi\rangle\,dt +\int_0^T \langle \sum_{k=2}^n \bar c_k , \psi\rangle \,dt  & \text{ if }n\geq 2\text{ and }\psi \in \mathcal D( Q_n) ,
\end{cases}
}
where $Q_n=B_{n}(0)\times (0,T)$ and
\[
\bar c_{k}(t)=-\int K_{ij}(x-y)  (\th_{k} (-y)-\th_{k-1} (-y))( u_iu_j)(y,t) \,dy.
\]
Given $\bar p$ as above, it is possible to show $ \bar p$ satisfies Definition~\ref{def.local.pressure2} (again, see \cite{BT7}).  Generally speaking, if we say ``let $\bar p$ satisfy the distributional local pressure expansion for a given $u$,'' then we mean $\bar p$ is constructed as above. We note that we can apply this construction to any matrix $f$ with entries $f_{ij}$ where the role of $u_iu_j$ is played by~$f_{ij}$.

While the above gives a distributional local pressure expansion, when $u$ has better local integrability,  a similar construction  gives a function that satisfies the local pressure expansion   in the sense of Definition~\ref{def.local.pressure}. The details of the construction are contained in \cite{BT7}; see also \cite{KiSe, KMT}.

In the remainder of this section we establish several properties of the distributional local pressure expansion. 
The first of these concerns the distributional local pressure expansion applied to matrices of the form $(c_iu_j)_{i,j}$ with $c_i$ time-dependent functions which are constant in the space variable, $u(x,t)$ is a vector field in $L_\uloc^2(0,T)$ and $u$ is divergence-free. In particular, the gradient of the image of such a matrix by the local pressure expansion is zero.

\begin{lemma}\label{lemma.zero} 
Assume $u\in L^2_\uloc(0,T)$ is divergence-free.  Let $\pi$ be given by \eqref{def.pressure.dist}, where $u_iu_j$ is replaced by $f_{ij}= c_i \, u_j$ and $c_i=c_i(t)$ are time dependent, spatially constant, bounded functions.  Then, $\nb \pi = 0$ in $\mathcal D'(\R^3\times (0,T))$.   
\end{lemma}
To prove this, we will need a version of Bogovskii's map from~\cite{Bog}.
\begin{lemma}[The Bogovskii map]
\label{lemma.bogovskii}
Let $\Om$ be a bounded Lipschitz domain in $\R^n$, 
where $2\le n <\infty$. There is a linear map $\Psi$ that maps a scalar $f \in L^q(\Om)$ 
with $\int_\Om f = 0$, $1<q<\infty$, to a vector field $v=\Psi f \in W^{1,q}_0(\Om;\R^n)$
and
\begin{equation*}
\div v= f, \quad \norm{v}_{W^{1,q}_0( \Om)} \le c(\Om,q) \norm{f}_{L^q( \Om)}.
\end{equation*}
The map $\Psi$ is independent of $q$ for $f \in C_c^\infty( \Om)$. 
\end{lemma}
The above statement is taken from \cite{Tsai-book} which references the more detailed treatment in~\cite{galdi}.  In the application below, the domains $\Omega$ are shells $\{  x:r\leq |x|\leq 2r  \}$. For such domains, we have $c(\Om,2) = C r$.

\begin{proof}[Proof of Lemma~\ref{lemma.zero}]
Let $\psi\in \mathcal D(\R^3\times (0,T))^3$, and choose  $R$ so that  $B=B_R(0)$ is a ball containing the support of $\psi$ at all times $t\in (0,T)$. 
Since $c_i(t)u_j(t)\in L^2$ for almost every $t$, we have 
\[
\int_0^T \langle    \pi ,  \nb \cdot \psi \rangle\,dt
=\int_0^T\int     G_{ij}^B (c_iu_j)     \nb \cdot \psi \,dx\,dt,
\]
where we have used that $\nb\cdot \psi$ is mean zero to eliminate the constant and skew-adjointness of the Riesz transforms in $L^2$ to move $R_iR_j$ from the test function to the localized quadratic term in the near part of the pressure expansion.

Fix a smooth cut-off $\ga$, which equals $1$ on $B$ and
zero off of $B_{2R}(0)$. 
Let $\ga_\e (x)= \ga(\e x)$ so that $\ga_\e=1$ on $B_{R/\e}(0)$.  Let $\eta_\e$ be a space-time mollifier and
$\td u^\e =  \ga_\e (\eta_\e * u_\e)$.  Then, $\td u^\e \to u$ in $L^p_\loc(\R^3\times (0,T))$ for every $1\leq p\leq 2$. 
Furthermore, $\td u^\e\in L^2_\loc(\R^3\times (0,T))$ is smooth, bounded and compactly supported.  However, $\nb \cdot  \td u^\e  =  ( \eta_\e * u) \cdot \nb \ga_\e$ in general does not vanish.
Applying  Lemma~\ref{lemma.bogovskii}, we obtain that
\[
u^\e = \td u^\e - \Psi^\e
\]
is divergence-free,
where
we have denoted $\Psi^\e = \Psi( ( \eta_\e * u) \cdot \nb \ga_\e)$.
Note that Lemma~\ref{lemma.bogovskii} is being applied to the domain $\supp \nb \ga_\e$, and, and thus $\Psi^\e =0 $ on~$B_{R/\e}(0)$.  We furthermore have 
\[
\|  \Psi^\e \|_{H^1}\leq C \e^{-1} \| (\eta_\e * u) \cdot \nb \ga_\e \|_{L^2} \leq C  \| \eta_\e * u \|_{L^2},
\]
where we have used the estimate from Lemma \ref{lemma.bogovskii} and the fact that
$\|\nb \ga_\e\|_{L^\I}\lesssim \e$.
Let $\pi^\e$ satisfy the local pressure expansion as a distribution with $u_iu_j$ replaced by $f_{ij} = c_i\td u_j^\e=c_iu_j^\e + c_i\Psi_j^\e$.
 Following ideas in \cite{BT7,KiSe,KwTs}, it is possible to show $\nb \pi^\e\to \nb \pi$ in $\mathcal D'(\R^3\times (0,T))$ (for reference, see \cite[Lemma~5.2]{BT7}).  
Now,  the functions $c_i u_j^\e$ all have compact support and  belong to~$L^2$. So,
the modified formula for the Riesz transforms, i.e., the local pressure expansion, agrees with the classical formulas for the Riesz transforms modulo an additive constant. This constant does not appear when tested against $\nb \cdot \psi$ since $\nb\cdot \psi$ has mean zero. This leads to 
\EQ{\label{eq.2.2}
\int    \pi^\e     \nb \cdot \psi \,dx &=  \int G_{ij}^B (c_i u_j^\e )\,\nb \cdot \psi \,dx + \int G_{ij}^B (c_i\Psi_j^\e )\,\nb \cdot \psi \,dx
\\&= \int G_{ij} (c_i u_j^\e )\,\nb \cdot \psi \,dx + \int G_{ij}^B (c_i\Psi_j^\e )\,\nb \cdot \psi \,dx
\\&= \int c_i u_j^\e G_{ij}  \nb \cdot \psi  \,dx + \int G_{ij}^B (c_i\Psi_j^\e )\,\nb \cdot \psi \,dx
\\&= c_i \int u_j^\e     \partial_j\partial_i (-\Delta)^{-1} ( \nb \cdot \psi ) \,dx+ \int G_{ij}^B (c_i\Psi_j^\e )\,\nb \cdot \psi \,dx,
}
for every time $t\in (0,T)$.
We will show that the first term on the right-hand side vanishes while the second term, when considered inside a time integral, goes to zero as $\e\to 0$.
Since all the terms are smooth in the first integral on the last line of \eqref{eq.2.2} and $u_j^\e$ is compactly supported, we may integrate by parts and  use that $u^\e$ is divergence-free to conclude that 
\[
 c_i \int u_j^\e    \partial_j\partial_i (-\Delta)^{-1} ( \nb \cdot \psi ) \,dx=0.
\]
On the other hand, 
\EQ{
&\int_0^T\int G_{ij}^B (c_i\Psi_j^\e )\,\nb \cdot \psi \,dx  \,dt
\\&\leq \int_0^T \int R_i R_j( c_i\Psi_{j}^\e \th) \nb\cdot \psi \,dx\,dt +C \int_0^T\int \nb\cdot \psi \int_{|y|\geq R}\frac 1 {|y|^4} |c_i| |\Psi_j^\e| \,dy\,dx\,dt
,
}
where $C$ is independent of~$\e$.  Note that $\supp \th$ is compact and independent of~$\e$. On the other hand, $\Psi_j^\e(y) = 0$ for $y\leq R/\e$. So, by taking $\e$ small enough, we can guarantee that $\th \Psi_j^\e \equiv 0$. Hence,
\[
\lim_{\e\to 0^+}\int_0^T \int R_iR_j( c_i\Psi_{j}^\e \th) \nb\cdot \psi \,dx\,dt  = 0.
\]
For the other term, and again because $\Psi_j^\e(y) = 0$ for $y\leq R/\e$, we have 
\EQ{
\int_{|y|\geq R}\frac 1 {|y|^4} |c_i| |\Psi_j^\e| \,dy &\leq \int_{|y|\geq R/\e}\frac 1 {|y|^4} |c_i| |\Psi_j^\e| \,dy
\leq C \bigg(\frac \e R   \bigg)^{5/2} \|\Psi_j^\e \|_2 
\\&\leq C \e^{3/2} \| (\eta_\e * u) \cdot\nb \ga_\e \|_2
,
}
where we used Lemma~\ref{lemma.bogovskii}. Note that 
$(\eta_\e * u)\cdot\nb  \ga_\e = (\eta_\e * (u\chi_{B_{2R/\e}}) )\cdot \nb \ga_\e $, and
thus
\EQ{
C \e^{3/2} \| (\eta_\e * u) \cdot\nb \ga_\e \|_2&\leq C \e^{3/2} \| \nb \ga_\e \|_\I \| \eta_\e *(u\chi_{B_{2R/\e}})\|_2
\\&\leq C \e^{5/2}   \| \eta_\e * (u  \chi_{B_{2R/\e}})\|_2
\leq C \e^{5/2}  \|u\|_{L^2(B_{2R/\e})}.
}
Therefore,
\EQ{
 &
 \int_0^T\int \nb\cdot \psi \int_{|y|\geq R}\frac 1 {|y|^4} |c_i| |\Psi_j^\e| \,dy\,dx\,dt
 \\&\indeq
 \leq C \e^{5/2} \esssup_{0<s<T} \|u\|_{L^2(B_{2R/\e})} \int_0^T\int \nb \cdot \psi \,dx\,dt 
 \leq C \e \|u\|_{L^2_\uloc(0,T)}.
}
These observations imply
\EQ{
\int_0^T\int G_{ij}^B (c_i\Psi_j^\e )\,\nb \cdot \psi \,dx  \,dt \to 0\mbox{ as }\e\to0.
}
It follows that $\nb \pi^\e \to 0$ as $\e\to 0$ in $\mathcal D'(\R^3\times (0,T))$.
Since we also know $\nb \pi^\e \to \nb \pi$ in $\mathcal D'(\R^3\times (0,T))$,
we conclude that $\nb \pi = 0 $ in $\mathcal D'(\R^3\times (0,T))$.
\end{proof}

\begin{lemma}\label{lemma.harmonic}Assume $u_0\in L^1_\loc$ is divergence-free and $u\in L^2_\uloc(0,T)$ is a weak solution with an associated pressure~$p$.   
Let $  \bar p$ be given by \eqref{def.pressure.dist} for~$u$.   
Then, $p_h=p-\bar p$ is harmonic.
\end{lemma}
\begin{proof}
Let $\psi\in \mathcal D(\R^3)$ be given.  We know that
\[
\Delta p = - \partial_i\partial_j (u_i u_j)
\]
in~$\mathcal{D}'$.
It suffices to show  that
\[
\Delta \bar p = - \partial_i\partial_j (u_i u_j)
\]
in~$\mathcal D'$.

Let $\eta $ be a spatial mollifier, and let $F_{ij}^\e = \ga_\e \eta_\e *(u_iu_j)$, where $\ga_\e$ is defined in the proof of Lemma~\ref{lemma.zero}.  Let $p^\e$ satisfy the distributional local pressure expansion for the matrix with the entries~$F_{ij}^\e$.  Then, $p^\e$ also satisfies the local pressure expansion by skew symmetry of the Riesz transforms in $L^2$ and the fact that $F_{ij}^\e$ is in $L^2$ since it is bounded and compactly supported.

Note that $\nb p^\e \to \nb \bar p$ (again, see \cite[Lemma~5.2]{BT7}) and $F_{ij}^\e \to F_{ij}$ (this is obvious by properties of mollifiers) all in the  distributional sense. Furthermore, $F_{ij}^\e$ are compactly supported, so 
\EQ{\label{eq.6.19.19}
\int p^\e \phi \,dx = \int (G_{ij}^B F_{ij}^\e )\phi \,dx = \int (G_{ij} F_{ij}^\e) \phi \,dx,
}
for any $\phi \in \mathcal D$, so that $\phi$ has mean zero,
where $B$ is a ball containing the support of~$\phi$. 
For $\psi$ as given, 
we have
\EQ{ 
& \langle \bar p, \Delta  \psi \rangle   + \int u_i u_j \partial_i\partial_j \psi \,dx
\\& = \langle \bar p -p^\e ,\Delta \psi \rangle + \int (F_{ij}-F_{ij}^\e)\partial_i\partial_j \psi \,dx
+ \int ( p^\e \Delta \psi+  F_{ij}^\e \partial_i\partial_j \psi )\,dx.
}
By the well known identity $\partial_i\partial_j \psi=R_iR_j\Delta \psi $ from  \cite{Stein2}, we have
\EQ{
\int   F_{ij}^\e \partial_i\partial_j \psi \,dx &=-\int F_{ij}^\e R_i R_j\Delta \psi \,dx
= -\int G_{ij} F_{ij}^\e \Delta \psi \,dx
=-\int p^\e \Delta \psi \,dx,
}
by~\eqref{eq.6.19.19}.
Since  $\Delta \psi = \nb\cdot \nb \psi$ and   $\nb p^\e \to \nb \bar p$ in the sense of distributions,  we get
\[
\langle \bar p-p^\e , \Delta  \psi \rangle \to 0.
\]
By the convergence of $F_{ij}^\e$ to $F_{ij}$ in $\mathcal D'$, we obtain
\[
\int(F_{ij}-F_{ij}^\e)\partial_i\partial_j \psi \,dx \to 0.
\]
It follows that 
$\langle \bar p, \Delta  \psi \rangle  +\int u_iu_j \partial_i\partial_j \psi \,dx = 0$.
\end{proof}

\section{The structure of weak solutions}\label{sec.proof1}

We begin by recalling some details from \cite{Kukavica} and~\cite{KuVi}.  Let $u$ be a weak solution to the Navier-Stokes equations with pressure $p$, and let $\bar p$ satisfy the distributional local pressure expansion. By Lemma~\ref{lemma.harmonic}, $p_h=p-\bar p$ is harmonic.
As in \cite{Kukavica}, we see that $\partial_{lk} p_h=0$ for all $1\leq l,k\leq 3$.  Hence, for $1\leq k\leq 3$, $\partial_k p_h$ is a distribution that only depends on~$t$.

Fix an element $\be \in C_0^\I(\R^3)$ with $\int \be \,dx = 1$. Let  
\EQ{
\phi_k(t)&= \int u_k(x,t) \be(x)\,dx - \int u_{0k} (x)\be (x)\,dx - \int_0^t\int u_k(x,s)\Delta \be (x)\,dx\,ds 
\\& -\int_0^t \int u_k (x,s) u_j(x,s) \partial_j \be (x)\,dx\,ds - \lim_{\e\to 0^+}\int_{\e}^t \langle  \bar p, \partial_k\be\rangle \,ds,
}
for $t>0$.  
That right-hand side is finite valued is clear for the terms involving~$u$. For the pressure term, note that   $\langle \bar p,\partial_k\be \rangle\in L^1(0,T)$ (this is a consequence of the definition of $\bar p$; see the first paragraph in \cite[Proof of Lemma~3.1]{BT7}). This means that the limit is unnecessary when considering $\bar p$ and we subsequently remove it. 
We claim that $\phi_k'(t)=\partial_k p - \partial_k \bar p$ in $\mathcal D'(0,T)$.  With $\lambda\in \mathcal D(0,T)$, we have
\EQ{
-\langle \phi_k',\la \rangle &= \int_0^T\int u_k(x,t) \be(x) \la'(t)\,dx\,dt - \int_0^T \int u_{0k} (x)\be (x) \la'(t)\,dx\,dt
\\&- \int_0^T \la'(t) \int_0^t\int u_k(x,s)\Delta \be (x) \,dx\,ds \,dt
\\& -\int_0^T \la'(t)\int_0^t \int u_k (x,s) u_j(x,s) \partial_j \be (x)\,dx\,ds \,dt
\\& - \int_0^T \la'(t)   \int_0^t \langle  \bar p, \partial_k\be\rangle \,ds\,dt.
}
Note that 
\[
\int_0^T \int u_{0k} (x)\be (x) \la'(t)\,dx\,dt = 0
\]
since $ \int u_{0k} (x)\be (x)  \,dx$ is independent of~$t$.
Integrating by parts and using the fundamental theorem of calculus, a.e.~in $t$, we have 
\EQ{
-\langle \phi_k',\la \rangle &= \int_0^T\int u_k(x,t) \be(x) \la'(t)\,dx\,dt
+\int_0^T \int u_k(x,t)\Delta \be (x)\la(t) \,dx\,dt
\\&+\int_0^T  \int u_k (x,t) u_j(x,s) \partial_j \be (x)\la(t)\,dx \,dt
+\int_0^T    \langle  \bar p, \partial_k\be \la \rangle(t) \,dt.
}
Examining the right-hand side above, we see that it contains all terms from \eqref{eq.NSE} tested against $\be \la$ except for the pressure. Hence,
\[
-\langle \phi_k',\la \rangle =     \int_0^T    \langle  \bar p - p, \partial_k\be \la \rangle(t) \,dt,
\]
which shows that $\phi_k' = \partial_k (p - \bar  p)$ in~$\mathcal D'(0,T)$.  
 
Let $\td u(x,t)= u(y,t)+\phi(t)$, where $y = x-\Phi(t)$
and $\Phi(t)=\int_{0}^{t}\phi(s)\,ds$.  
Let $\nb \td p$ be the distribution defined by the map
\EQ{
&\psi \in \mathcal D(\R^3\times (0,T)) \mapsto
\int_0^T \langle  \nb \bar p, \psi_{\Phi(t)}   \rangle \,dt,
}
where $( \nb \cdot \psi)_{\Phi(t)} = \nb \cdot \psi(\cdot + \Phi(t)  )$. It is not difficult to see that this defines a distribution because, for each time $t$, the translation $\psi_{\Phi(t)}$ belongs to $\mathcal D(\R^3)$, and so the action of $\nb \bar p$ is meaningful at every~$t$.   

By straightforward computations,
\EQ{
&\partial_t \td u(x,t ) = \partial_t u(y,t) +\phi'(t) - \phi_k(t)\partial_k u(y,t),
\\&\Delta_x \td u(x,t)=\Delta_y u(y,t),
\\&\td u\cdot \nb \td u (x,t)= u\cdot \nb u(y,t) + \phi_k(t) \partial_k u(y,t),
}
as distributions. 
Therefore,
\EQ{
&\partial_t \td u(x,t) - \Delta \td u(x,t) +\td u \cdot \nb \td u(x,t) +\nb \td p(x,t) -\phi'(t)
\\&=\partial_t u(y,t) - \Delta u(y,t ) + u\cdot \nb u (y,t) +\nb \bar p(y,t) =0,
}
where all the equalities are understood in the sense of distributions.
We also have $\nb \bar p +\phi'(t)=\nb p$ as distributions. Thus, $\td u$ and $\td p$ solve \eqref{eq.NSE} in the sense of distributions.

We will now establish several properties of $\td u,\,\td p$, and~$\phi$.

\begin{lemma}\label{lemma.local.pressure.expansion}
The pressure $\td p$  satisfies the distributional local pressure expansion for~$\td u$. 
\end{lemma}

\begin{proof}
When we write down the local pressure expansion for $\td u$ we obtain the  distributional local pressure expansion for $u(\cdot + \Phi(t))$ plus the   terms
\[
 G_{ij}^B ( \phi_i(t) u_j (\cdot +\Phi(t),t)) + G_{ij}^B ( u_i( \cdot +\Phi(t),t) \phi_j(t)) + G_{ij}^B ( \phi_i(t) \phi_j(t)).
\]
By Lemma~\ref{lemma.zero}, these are all equal to zero. Since the distributional local pressure expansion  for $u(\cdot + \Phi(t))$ is just $\nb \td p$, and the proof is complete.
\end{proof}

\begin{lemma}\label{lemma.bound.phi} 
If $(u,p)$ is a weak solution in the sense of Definition~\ref{def.weak.sol}   for some divergence-free $u_0\in L^1_\loc$, then $u(t)\to u_0$ in $L^1_\loc(\R^3)$ (so also in $\mathcal D'(\R^3)$),
$\phi\in L^\I[0,T)$, and $\lim_{t\to 0^+}\phi(t)=0$. 
If, additionally, for every $w\in L^2$ with compact support we have that
\[
t\mapsto \int w(x) u(x,t)\,dx
\]
is continuous on $[0,T)$, then so is~$\phi$.
\end{lemma}

\begin{proof}
We first show $u(t)\to u_0 \in \mathcal D'$. Our argument is very similar to \cite[Proof of Lemma~3.1]{KuVi}. With $\psi\in \mathcal D(\R^3)$, we have
$\int u_k(\cdot,t) \psi_k \,dx \in L^{1}([0,T))$ because $u\in L^2_\uloc(0,T)$.
Let $\tau$ be in the Lebesgue set of $\int u_k(\cdot,t) \psi_k \,dx$. Following \cite[Proof of Lemma~3.1]{KuVi}, we obtain by the Lebesgue Differentiation Theorem that  
\EQ{
&
\int u_k(x,\tau) \psi_k (x)\,dx - \int u_{0k}(x)\psi_k(x)\,dx
\\&\indeq
= \int_0^\tau \int u_k( \Delta \psi_k +u_j\partial_j \psi_k)\,dx\,dt + \int_{0+}^\tau \langle p ,\partial_k \psi_k\rangle \,dt.
}
Note that
\[
\int_0^\tau \int
\Bigl(
|   u_k \Delta\psi_k| + |u_ku_j \partial_j \psi_k|
\Bigr)
\,dx\,dt \to 0,
\]
as $\tau \to 0$ because $u\in L^2_\uloc(0,T)$. On the other hand, our assumption that $\int_{0+}^T \langle p,\partial_k \psi\rangle\,dt$ exists for every $\psi\in C_0^\I(\R^3\times [0,T))$ implies
\[
 \int_{0+}^\tau \langle p ,\partial_k \psi_k\rangle \,dt \to 0,
\]
as $\tau \to 0$.
Therefore,  $u\to u_0$ in $\mathcal D'$ and, moreover, $u\to u_0\in L^1_\loc$.

Let $\be \in C_0^\I(\R^3)$ be as above. We estimate $\phi\in L^\I[0,T)$ by bounding each term on the right-hand side of the equation
\EQ{\label{eq.6.19.19.2}
\phi_k(t)&= \int u_k(x,t) \be(x)\,dx - \int u_{0k} (x)\be (x)\,dx - \int_0^t\int u_k(x,s)\Delta \be (x)\,dx\,ds 
\\& -\int_0^t \int u_k (x,s) u_j(x,s) \partial_j \be (x)\,dx\,ds - \lim_{\e\to 0}\int_\e^t \langle \bar p,\partial_k\be \rangle \,ds,
}  
for a.e.~$t\in [0,T)$.  
Note that $u_0\in L^1_\loc$, so the first term on the right-hand side is bounded. 
The next two terms are finite because $u\in L^2_\uloc(0,T)$. 
For the pressure, we note that $ \bar p$ satisfies  the local pressure expansion as a distribution for $u$, so we can write $  \bar p = \bar p_\text{near}+  \bar p_\text{far}$.   For the local part, we have 
\[
\int_{0+}^t \int u_iu_j \th_R  R_iR_j \partial_k \be(x)\,dx\,ds  \leq C  \|u\|_{L^2_\uloc(0,T)}^2,
\]
where $R$ is large enough so that $\supp\be \subset B_R(0)$; we have used that $\partial_k \be$ is Lipschitz to deplete the singularity in the singular integral to obtain $\| R_iR_j \partial_k \be\|_\I <\I$.
The far-field part of the pressure is controlled as in~\cite{BT7,KiSe,KwTs}. 

We have already shown that most of the terms that make up $\phi$ vanish at $t=0$ when we established that $u\to u_0$ in~$\mathcal D'$. Consequently, to prove $\phi(t)\to 0$ as $t\to 0$, it suffices to note that  $\langle \bar p,\partial_k\be \rangle\in L^1(0,T)$ (see the first paragraph in \cite[Proof of Lemma~3.1]{BT7}). 

Finally, the
continuity is obvious.
\end{proof}
 
We are now ready to prove Theorem~\ref{thrm.structure}.

\begin{proof}[Proof of Theorem~\ref{thrm.structure}]
Let $\td u(x,t)= u(x-\Phi(t),t)+\phi(t)$ and $\td p (x,t)=  \bar p(x-\Phi(t),t)$, where $\phi$ is as above.  By Lemmas \ref{lemma.local.pressure.expansion} and \ref{lemma.bound.phi}, $  \td p$ satisfies the distributional local pressure expansion, $\phi \in L^\I(0,T)$, and $\phi(t)\to 0$ as $t\to 0$.

It remains to show that $(\td u,\td p)$ is a weak solution. We know that $(\td u,\td p)$ solves \eqref{eq.NSE} distributionally.   
Since $\phi \in L^\I(0,T)$, we   have $\td u \in L^2_\uloc(0,T) $.

We need to prove that the pair $(\td u ,\td p)$ is weak in the sense of Definition~\ref{def.weak.sol}. At our disposal we have the distributional local pressure expansion and convergence in $\mathcal D'(\R^3)$ to the initial data (the latter is from Lemma~\ref{lemma.bound.phi}).
Consider
\EQ{
\int_K   (\td u(x,t)-u_0(x) ) \psi\,dx  
 &= \int    (u(x-\Phi(t),t)  - u_0(x-\Phi(t)))\psi \,dx
\\&+\int   (u_0(x-\Phi(t))-u_0(x)) \psi  \,dx + \int \phi(t)\psi \,dx.
}
The third term on the right-hand side clearly vanishes as $t\to 0$.  The second term vanishes as $t\to 0$ by continuity of translations in $L^1$ and $\Phi(t)\to 0$ as $t\to 0$.   The first vanishes after making a change of variables, noting that the transformed region of integration is still compact and using $u\to u_0$ in~$L^1_\loc$.

We now check that 
\EQ{\label{7.26.19}
\lim_{\e\to 0} \int_\e^T \langle \td p ,\partial_k \psi_k\rangle \,dt = \int_{0+}^T \langle \td p ,\partial_k \psi_k\rangle \,dt
}
exists.
This follows from the Dominated Convergence Theorem (DCT) once we note that $| \langle \td p (t),\partial_k \psi_k\rangle| \in L^1(0,T)$ whenever $u\in L^2_\uloc(0,T)$ (see the first paragraph in \cite[Proof of Lemma~3.1]{BT7}). 

We finally check that  \eqref{eq.weakForm} is satisfied by $(\td u,\td p)$. Let $\psi\in \mathcal D'(\R^3\times [0,T)  )$. Using the Lebesgue Differentiation and the DCT, it is possible to show for a.e.~$\tau\in (0,T)$ that 
\[
\int_\tau^T \int \td u_k(\partial_t \psi_k +\Delta \psi_k +\td u_j\partial_j \psi_k)  \,dx\,dt +\int_\tau^T \langle \td p ,\partial_k\psi_k\rangle \,dt = - \int \td u_k(\cdot, \tau) \psi_k(\cdot,\tau)\,dx.
\]
Thus we may choose a sequence of times  $\{\tau_n\}$ so that $\tau_n\to 0^+$ and
\[
\int_{\tau_n}^T \int \td u_k(\partial_t \psi_k +\Delta \psi_k +\td u_j\partial_j \psi_k \,dx\,dt +\int_{\tau_n}^T \langle \td p ,\partial_k\psi_k\rangle  \,dt= - \int \td u_k(\cdot, \tau_n) \psi_k(\cdot,\tau_n)\,dx.
\]
The right-hand side of this expression converges to $-\int u_{0k}(x) \psi(\cdot,x)\,dx $, while, by the DCT, the left-hand side converges to the left-hand side of~\eqref{eq.weakForm}.
\end{proof}

\section{Sufficient conditions for pressure expansion}\label{sec.proof2}

For a fixed $\be$,
denote $\beta_R(x)= {R^{-3}} \beta(x/R)$,
and let $\phi_R$ be the vector with components
\EQ{
\phi_{R,k}(t)&= \int u_k(x,t) \be_R(x)\,dx - \int u_{0k} (x)\be_R (x)\,dx - \int_0^t\int u_k(x,s)\Delta \be_R (x)\,dx\,ds 
\\& -\int_0^t \int u_k (x,s) u_j(x,s) \partial_j \be_R (x)\,dx\,ds - \lim_{\e\to 0}\int_\e^t\langle  \bar p(x,s),\partial_k\be_R(x)\rangle\,ds,
}
for $k=1,2,3$.

\begin{lemma}Under the assumptions of Theorem~\ref{thrm.sufficient}, $\phi_R(t)\to 0$ in $L^1(0,T)$ and $\Phi_R(t) =\int_0^t\phi_R(s)\,ds\to 0$ in $L^\I(0,T)$ as $R\to \I$.
\end{lemma}
 
\begin{proof}
We first show that $\phi_R(t)$    vanishes in~$L^1(0,T)$. For the first term, we have
\EQ{
\left|
\int_0^T \int u_k \be_R (x) \,dx\,dt
\right|
&\leq C\| \be_R \|_{L^\I} R^3 \bigg(\frac 1 {R^3}  \int_0^T\int_{B_R(0)} |u|^2 \,dx\,dt  \bigg)^{1/2}
\\& \leq C  \bigg( \frac 1 {R^3}\int_0^T\int_{B_R(0)} |u|^2\,dx\,dt  \bigg)^{1/2} \to 0,
}
as $R\to \I$ by assumption.
We estimate the next three terms similarly as
\EQ{
\int_0^T \int |u_{0k} | \be_R \,dx\,dt &\leq   C  \frac T {R^3} \int_{B_R(0)} |u_0| \,dx ,
}
\EQ{
\left|\int_0^T \int_0^t\int u_k(x,s)\Delta \be_R (x)\,dx\,ds \,dt
\right|
&\leq C T \frac {R^{3}} {R^5} \bigg(\frac 1 {R^3} \int_0^T\int |u|^2 \,dx\,dt \bigg)^{1/2},
}
and 
\EQ{
\left|
\int_0^T    \int_0^t \int u_k (x,s) u_j(x,s) \partial_j \be_R (x)\,dx\,ds			\,dt
\right|
&\leq C\frac {R^3} {R^4}   \frac T {R^3} \int_0^T\int |u|^2 \,dx\,dt  .
}
These all vanish as $R\to \I$.

For the near-field part of the pressure, note that $\partial_k \be_R (x) = {R^{-4}}( \partial_k \be) (x/R)$.
Then,
\EQ{
&\int_0^T \lim_{\e\to 0^+}\int_\e^t \int u_i u_j \th_R R_i R_j(\partial_k \be_R)\,dx\,ds\,dt
\\&\leq \frac T {R^4} \int_0^T\int \bigg| u_i u_j \th_R \,\pv\, \int K_{ij} (x-y)((\partial_k \be) (y/R) - ( \partial_k \be) (x/R ))\,dy \bigg|dx\,dt
\\&\leq  \frac T {R^4} \int_0^T\int\bigg| u_i u_j \th_R\,\pv\, \int_{|x-y|\leq 1} K_{ij}(x-y) ((\partial_k \be) (y/R) - ( \partial_k \be) (x/R ))\,dy\bigg| dx\,dt
\\&+ \frac T {R^4} \int_0^T\int \bigg| u_i u_j \th_R  \int_{|x-y|>1} K_{ij}(x-y) (\partial_k \be) (y/R)  \,dy \bigg|dx\,dt,
}
where we used that $K_{ij}$ is mean zero on spheres centered at the origin.
Since $\nb \be $ is Lipschitz, we have
\EQ{
&\frac T {R^4} \int_0^T\int \bigg| u_i u_j \th_R \,\pv\, \int_{|x-y|\leq 1} K_{ij} (x-y)((\partial_k \be) (y/R) - ( \partial_k \be) (x/R ))\,dy\bigg| dx\,dt
\\&\leq T
\frac { C}   {R^4} \int_0^T\int | u_i u_j\th_R| \int_{|x-y|\leq 1} \frac {|x/R-y/R|}{|x-y|^3}\,dy
\\&\leq C T \frac {1}   {R^5} \int_0^T\int | u_i u_j\th_R|   \,dx\,dt
\\&	\leq C T R^{-2} \|u\|_{L^2_\uloc(0,T)}^2 				\to 0 \, \text{as $R\to \I$}.
} 
On the other hand,
\EQ{
& \frac T {R^4} \int_0^T\int\bigg| u_i u_j \th_R  \int_{|x-y|>1} K_{ij}(x-y) (\partial_k \be) (y/R) \,dy\bigg|dx\,dt
\\&\leq \frac T {R^4} \int_0^T\int |u_i u_j \th_R|  \|   |x|^{-3} \|_{L^{p'}(|x|>1)} \| (\partial_k \be)(\cdot/R) \|_{L^p} \,dx\,dt,
}
where $p$ and $p'$ are H\"older conjugates. Note that 
\[
 \| (\partial_k \be)(\cdot/R) \|_{L^p} = R^{3/p} \| (\partial_k \be)(\cdot ) \|_{L^p}.
\]
Taking $3<p<\I$, we obtain 
\EQ{
& \frac T {R^4} \int_0^T\int \bigg| u_i u_j \th_R  \int_{|x-y|>1} K_{ij}(x-y) ((\partial_k \be) (y/R) - ( \partial_k \be) (x/R ))\,dy\bigg| dx\,dt
\\&\leq  C T\frac {R^{3/p}}{R^4} \int_0^T\int | u_i u_j \th_R|   \,dx\,dt
\\&\leq C T  R^{3/p-1} \|u\|_{L^2_\uloc(0,T)}^2  \to 0  \,\text{as $R\to \I$}.
}

 For the far-field part, we have, for $x\in B_R$ and $y\in B_{2R}^c$,
\EQ{
\int_0^T |\bar p_{\text{far}} (x,t)| \,dt &\leq  R \int_0^T   \int_{|y|>2R} \frac 1 {|y|^4} |u(y,t)|^2\,dy\,dt
\\&\leq C R\sum_{i=1}^\I \int_0^T   \int_{2^i R<|y|\leq 2^{i+1}R} \frac 1 {|y|^4} |u(y,t)|^2\,dy\,dt
\\&\leq C R\sum_{i=1}^\I  \frac 1 {R^4 2^{4i}}  \int_0^T \int_{2^i R<|y|\leq 2^{i+1}R}  |u(y,t)|^2\,dy\,dt 
\\&\leq C\|u \|_{L^2_\uloc (0,T)}^2.
}
Thus,
\[
\int_0^T \int_\e^t \int \bar p_{\text{far}}(x,s)\partial_k\be_R(x)\,dx\,ds\,dt \leq CT \|u \|_{L^2_\uloc (0,T)}^2  \frac 1 {R^4} R^3 \to 0  \,\text{as $R\to \I$}.
\]

We also need to prove $\Phi_R(t)\to 0$. Above, we have proven $\int_0^T |\phi_R(s) |\,ds\to 0$. Note that $\Phi_R(t) =\int_0^t \phi_R(s)\,ds$.  So, for all $t\in [0,T]$,
\[
|\Phi_R(t)|\leq \int_0^t |\phi_R(s)|\,ds \leq  \int_0^T |\phi_R(s) |\,ds\to 0,
\]
implying $\Phi_R\to 0$ in $L^\I ([0,T])$.

\end{proof}

\begin{proof}[Proof of Theorem~\ref{thrm.sufficient}]
Fix $\beta \in C_c^\I$ so that $\int \be \,dx=1$, and let $\beta_R(x) = R^{-3} \beta(x/R)$. Then $\int \be_R \,dx = 1$ for every $R>0$.   Therefore, for each $R>0$, the discussion at the beginning of Section~\ref{sec.proof1} applies, and we define $\phi_R = \phi$, $u_R=\td u$ and  $p_R=\td p$ accordingly. 
We now prove that $u_R \to u$ in~$\mathcal D'$.  Fix $\psi\in \mathcal D$.  We know that
\[
u(x,t)= u_R(x-\Phi_R(t),t)+\phi_R(t)
\]
in $\mathcal D' (\R^3\times (0,T))$ and
$\phi_R(t)\to 0$ in $\mathcal D' (0,T)$.  We then have
\[
\int_0^T\int u(x,t)\psi(x,t) \,dx\,dt
= \int_0^T\int \big(u_R(x-\Phi_R(t),t)+\phi_R(t) \big)\psi\,dx\,dt.
\]
Focusing on the last term, we see that
\[
\lim_{R\to \I}  \bigg| \int_0^T\int \phi_R(t)\psi\,dx\,dt\bigg| \leq \lim_{R\to \I} \bigg\|\int \psi(x,\cdot) \,dx\bigg\|_{L^\I(0,T)} \int_0^T  | \phi_R(t) |\,dt  =0.
\]
Hence,
\[
u(x,t)= \lim_{R\to \I} u_R(x-\Phi_R(t),t)
\]
in~$\mathcal D'$.  

To see that   $\nb p_R\to \nb \bar p$ in $\mathcal D'$, recall that
\[
\langle \nb p_R(t) , \psi  \rangle = \langle  \nb \bar p(t), \psi_{\Phi_R(t)}\rangle,
\]
where the subscript indicates the translation in the $x$ variable by $\Phi_R(t)$ (this notation was introduced in Section~\ref{sec.proof1}).
Since $\Phi(t)\to 0$ in $L^\I(0,T)$, we have $\psi_{\Phi_R(t)} \to \psi$ in the topology on~$\mathcal D$ (e.g.~by continuity of translations in the $L^p$ norms).  Here, we take $\mathcal D$ to be $\mathcal D(\R^3)$ (we do not have $\psi_{\Phi_R(t)} \in \mathcal D(\R^3\times (0,T))$).
Therefore, $\nb p_R\to \nb p$ in~$\mathcal D'(\R^3)$.   So, for every $t\in (0,T)$, we have $ \langle  \nb \bar p(t), \psi_{\Phi_R(t)}\rangle \to \langle \nb \bar p (t), \psi \rangle$, and we may use the~DCT to conclude that
\[
\lim_{R\to \I} \int_0^T \langle \nb p_R(t) , \psi  \rangle \,dt = \int_0^T \langle \nb \bar p , \psi \rangle \,dt.
\]

We now have $u_R\to u$ and $\nb p_R\to \nb \bar p$ in $\mathcal D'(\R^3\times (0,T))$, and $u_R$ and $p_R$ solve the Navier-Stokes equations.  This property is inherited by the limit, implying that $u$ and $\bar p$ solve the Navier-Stokes equations in~$\mathcal D'$. This proves $\nb p = \nb \bar p$ in~$\mathcal D'$.
 \end{proof}

\begin{proof}[Proof of Theorem~\ref{thrm.sufficient2}] 
By Theorem~\ref{thrm.sufficient}, $u$ satisfies the local pressure expansion as a distribution. 
Note that $u\in L^{8/3} (0,T;L^4(B)))$ for any ball~$B$. Then $\th u_iu_j\in L^2$ at a.e.~time, implying 
\[
\int u_iu_j \th R_iR_j (\nb \cdot\psi )\,dx = \int R_iR_j(u_iu_j\th) \nb \cdot \psi\,dx,
\]
for almost every~$t$. This shows that the near part of the distributional local pressure expansion   agrees with the near   part of the local pressure expansion in $\mathcal D'(\R^3\times (0,T))$, which we needed to show.
\end{proof}

\section{Appendix}

For convenience we compare the various sufficient conditions for either the local pressure expansion or mildness. Recall the three conditions discussed in Section~1,
\begin{align} 
&\lim_{|x_0|\to \I} \int_0^{R^2}\int_{B_R(x_0)} |u|^2\,dx\,dt =0,  \label{cond.A}
\\&\lim_{R\to \I} \sup_{x_0\in \R^3} \frac 1 {R^3} \int_0^T\int_{B_R(x_0)} |u|^2\,dx \,dt = 0, \label{cond.B}
\\&\lim_{R\to \I} \frac 1 {R^3} \int_0^T\int_{B_R(0)} |u|^2\,dx \,dt = 0,
\label{cond.C}
\end{align}
which appear in \cite{JiaSverak-minimal}, \cite{LR}, and Section~1, respectively.
The strongest condition is   \eqref{cond.A} and the weakest is \eqref{cond.C}, as the following assertions show:
\begin{itemize}
\item \eqref{cond.A} implies~\eqref{cond.B}. Assume $u$ satisfies~\eqref{cond.A}. Then, for any $\e>0$, there exists $R_\e$ so that, if $|x_0|>R_\e$, we have 
\[
\int \int_{B_1(x_0)} |u|^2\,dx\,dt <\e.
\]
Then,
\EQ{
&\frac 1 {R^3} \int\int_{B_R(x_0)} |u|^2\,dx\,dt 
\\& \leq \frac 1 {R^3} \int\int_{B_R(x_0) \cap B_{R_\e}(0)}  |u|^2\,dx\,dt + \frac 1 {R^3} \int\int_{B_R(x_0)\cap B_{R_\e}(0)^c}  |u|^2\,dx\,dt
\\&\leq \frac { C R_\e^3} {R^3} \|u\|_{L^2(0,T)}^2 + C\e.
}
Clearly this can be made arbitrarily small by first taking $\e$ small and then taking $R$ large in comparison to~$R_\e$.
\item \eqref{cond.B} implies~\eqref{cond.C}. This is obvious. 
\item \eqref{cond.B} does not imply~\eqref{cond.A}. A counterexample is the characteristic function $f$ for the cylinder of radius $1$ oriented along the $x_1$ axis. Then
\[
\frac 1 {R^3} \int_0^T\int_{B_R(x_0)} |f|^2 \,dx\,dt \leq \frac {C} {R^2} \to 0,
\]
but if $\hat x_k$ is a sequence on the $x_1$-axis so that $|\hat x_k|\to \I$, we have 
\[
\int_{B_1(\hat x_k)} |f|^2\,dx\,dt = |B_1(\hat x_k)|, 
\]
which does not vanish.
\item \eqref{cond.C} does not imply~\eqref{cond.B}. Let $\hat x_k = (  2^k,0,0)$ and let $f(x,t) = \sum_{ k\in \N } \chi_{B_{k}(\hat x_k)}$. This clearly does not satisfy \ref{cond.B} since 
\[
\frac 1 {k^3} \int_0^T\int_{B_{k}( \hat x_k)} |f|^2\,dx\,dt = T,
\]
does not vanish as $k\to \I$.  However, if $2^k \leq R < 2^{k+1}$, then
\[
 \frac 1 {R^3}\int_0^T\int_{B_R(0)}|u|^2\,dx\,dt\leq  C \frac T {R^3} \sum_{1\leq i \leq \lceil   \log_2 R   \rceil } i^3 \leq  \frac T{2^{3k}}       (k+1)^4,
\]
which vanishes as $R\to \I$. Hence, $f$ satisfies~\eqref{cond.C}.
\end{itemize}

The examples above can be modified to ensure they are divergence-free. 

We conclude with a  remark on the assumptions on the initial data. It is easy to verify that the condition in Corollary~\ref{thrm.sufficient2} is implied if $u_0\in E^2$.  Hence, Corollary~\ref{thrm.sufficient2} generalizes the condition in \cite{JiaSverak-minimal} (i.e.,~\eqref{cond.A}).  In contrast, in \cite[Theorem~11.1]{LR}, the initial data is not mentioned.

\end{document}